\documentclass[11pt,reqno]{amsart}

\usepackage[T1]{fontenc}
\usepackage[utf8]{inputenc}
\usepackage{lmodern}
\usepackage{amsmath,amssymb,amsthm,mathtools}
\usepackage[hidelinks]{hyperref}
\usepackage{microtype}

\allowdisplaybreaks
\numberwithin{equation}{section}

\theoremstyle{plain}
\newtheorem{theorem}{Theorem}[section]
\newtheorem{corollary}[theorem]{Corollary}
\newtheorem{lemma}[theorem]{Lemma}
\newtheorem{proposition}[theorem]{Proposition}

\theoremstyle{definition}

\theoremstyle{remark}
\newtheorem{remark}[theorem]{Remark}

\newcommand{\T}{\mathbb T}
\newcommand{\Z}{\mathbb Z}

\newcommand{\eps}{\varepsilon}
\newcommand{\dd}{\,d}
\newcommand{\wh}{\widehat}
\newcommand{\norm}[1]{\left\lVert #1\right\rVert}
\newcommand{\abs}[1]{\left\lvert #1\right\rvert}

\hypersetup{
  pdftitle={An Endpoint Dini Form of Zhizhiashvili's Theorem for Rectangular Fourier Series},
  pdfauthor={Ushangi Goginava},
  pdfsubject={Multiple Fourier series and almost-everywhere rectangular convergence},
  pdfkeywords={Fourier series, rectangular partial sums, Pringsheim convergence, Zhizhiashvili theorem}
}

\begin{document}

\title[A Critical-Scale Extension of Zhizhiashvili's Theorem]
{A Critical-Scale Extension of Zhizhiashvili's Theorem for Rectangular Fourier Series}
\author[U. Goginava]{Ushangi Goginava}
\address{Department of Mathematical Sciences, United Arab Emirates University, P.O. Box No. 15551, Al Ain, Abu Dhabi, United Arab Emirates}
\email{zazagoginava@gmail.com}
\email{ugoginava@uaeu.ac.ae}
\date{}

\begin{abstract}
We address a long-standing endpoint problem arising from Zhizhiashvili's logarithmic modulus theorem for multiple Fourier series. We prove an endpoint Dini criterion for almost-everywhere Pringsheim convergence of ordinary rectangular partial sums. In Zhizhiashvili's theorem the logarithmic modulus is assumed with exponent strictly above the critical value; here this strict power margin is replaced by a summable endpoint Dini condition. As a consequence, one obtains double-logarithmic endpoint classes lying outside the range of the classical theorem. The proof reduces the endpoint smoothness assumption to the Kaczmarz--Kojima product-logarithmic coefficient criterion by weighted translation-difference estimates.
\end{abstract}

\maketitle
\enlargethispage{8pt}

\section{Introduction}

Almost-everywhere convergence of Fourier series has a fundamentally different nature in one and in several variables. In one dimension, the Carleson--Hunt theorem gives almost-everywhere convergence of the symmetric partial sums for every function in $L^p(\T)$, $p>1$ \cite{Carleson1966,Hunt1968}. In several variables, the geometry of the partial sums becomes part of the problem. The present note concerns ordinary symmetric rectangular sums and convergence in the sense of Pringsheim.

Throughout, $\T=\mathbb R/2\pi\mathbb Z$, and $dx$ denotes the normalized Haar measure on $\T^d$; thus, on the fundamental cube,
\[
 dx=(2\pi)^{-d}\,dx_1\cdots dx_d .
\]
For $f\in L^1(\T^d)$ put
\[
 \wh f(k)=\int_{\T^d} f(x)e^{-ik\cdot x}\dd x,
 \qquad k=(k_1,\ldots,k_d)\in\Z^d.
\]
The ordinary rectangular partial sums are
\begin{equation}\label{eq:rectangular-sums}
 S_{\mathbf n}f(x)=
 \sum_{\abs{k_1}\le n_1}\cdots\sum_{\abs{k_d}\le n_d}\wh f(k)e^{ik\cdot x},
 \qquad \mathbf n=(n_1,\ldots,n_d)\in\mathbb N_0^d.
\end{equation}
We say that $S_{\mathbf n}f$ converges to $f$ in the Pringsheim sense if
\[
 S_{\mathbf n}f(x)\longrightarrow f(x)
 \qquad\text{as }\min_{1\le j\le d}n_j\to\infty .
\]
This is a genuinely multi-parameter question. Fefferman's divergence theorem shows that the higher-dimensional theory contains divergence phenomena absent from the one-dimensional theory \cite{Fefferman1971}. It is therefore natural to ask which quantitative smoothness assumptions recover almost-everywhere convergence at the borderline of the logarithmic scale.

For $h=(h_1,\ldots,h_d)\in\T^d$, we write
\[
 \abs{h}=(h_1^2+\cdots+h_d^2)^{1/2},
\]
where $h_j\in[-\pi,\pi)$ is the representative near the origin. For $f\in L^2(\T^d)$ let
\begin{equation}\label{eq:modulus}
 \omega_2(f,\delta)=\sup_{\abs{h}\le\delta}\norm{f(\cdot+h)-f(\cdot)}_{L^2(\T^d)},
 \qquad 0<\delta<1 .
\end{equation}
Zhizhiashvili proved the following fundamental logarithmic criterion: if $d\ge2$, $f\in L^2(\T^d)$, and
\begin{equation}\label{eq:zhizhiashvili-classical}
 \omega_2(f,\delta)=O\left(\left(\log \frac1\delta\right)^{-\beta}\right)
 \qquad(\delta\downarrow0)
\end{equation}
for some $\beta>d/2$, then the rectangular Fourier sums of $f$ converge almost everywhere in the Pringsheim sense \cite{Zhizhiashvili1996}. The strict inequality $\beta>d/2$ is the decisive feature of this theorem. It leaves open the critical logarithmic exponent itself.

The main result of this paper gives a partial endpoint solution to this problem by replacing the missing logarithmic power margin with a Dini summability condition.

\begin{theorem}[Endpoint Zhizhiashvili--Dini criterion]\label{thm:main}
Let $d\ge2$ and $f\in L^2(\T^d)$. Suppose that, for some $\eta\in(0,1)$,
\begin{equation}\label{eq:main-dini}
 \int_0^\eta \omega_2(f,t)^2\left(\log\frac et\right)^{d-1}\frac{\dd t}{t}<\infty .
\end{equation}
Then
\[
 S_{\mathbf n}f(x)\longrightarrow f(x)
 \qquad\text{for almost every }x\in\T^d
\]
in the Pringsheim sense.
\end{theorem}

Theorem~\ref{thm:main} is a significant endpoint advance of Zhizhiashvili's theorem in the $L^2$ scale. The classical hypothesis \eqref{eq:zhizhiashvili-classical} with $\beta>d/2$ immediately implies \eqref{eq:main-dini}. More importantly, Theorem~\ref{thm:main} also includes moduli lying exactly at the critical power $d/2$, provided that the remaining endpoint contribution is summable. The following consequence records the simplest explicit form of this gain.

\begin{corollary}[Double-logarithmic endpoint class]\label{cor:double-log}
Let $d\ge2$, $f\in L^2(\T^d)$, and $\gamma>0$. If
\begin{equation}\label{eq:double-log-assumption}
 \omega_2(f,\delta)=O\left(
 \left(\log\frac e\delta\right)^{-d/2}
 \left(\log\log\frac{e^e}{\delta}\right)^{-1/2-\gamma}
 \right)
 \qquad(\delta\downarrow0),
\end{equation}
then $S_{\mathbf n}f(x)\to f(x)$ for almost every $x\in\T^d$ in the Pringsheim sense.
\end{corollary}

Corollary~\ref{cor:double-log} is not a restatement of Zhizhiashvili's result. It permits the critical logarithmic power $d/2$ and compensates only by a summable iterated logarithm. Thus it applies to moduli which may decay more slowly than every power
\[
 \left(\log\frac1\delta\right)^{-d/2-\eps},\qquad \eps>0,
\]
and hence are outside the range of \eqref{eq:zhizhiashvili-classical}.

The proof is short and structural. We use the Kaczmarz--Kojima theorem, which says that a product-logarithmic square summability condition on the Fourier coefficients implies almost-everywhere rectangular convergence. The new point is that the endpoint Dini condition \eqref{eq:main-dini} forces exactly this coefficient condition. A weighted lower bound for the oscillatory multipliers $e^{imt}-1$ converts one-coordinate translation estimates into one-coordinate logarithmic coefficient estimates. The arithmetic--geometric mean inequality then converts the resulting family of estimates into the Kaczmarz--Kojima product weight.

\section{Preliminaries}

All logarithms are natural. Constants denoted by $C$ may change from line to line and may depend on fixed parameters such as $d$ and $\eta$, but not on the frequency variable. We use Plancherel's theorem and uniqueness of trigonometric Fourier coefficients in their standard forms; see, for instance, \cite{Zygmund2002}.

For a coordinate vector $e_j$, define
\[
 \Delta_j(t)f(x)=f(x+te_j)-f(x),\qquad 1\le j\le d.
\]
We shall use the following classical convergence theorem.

\begin{theorem}[Kaczmarz--Kojima product-log theorem]\label{thm:KK}
Let $d\ge2$ and $F\in L^2(\T^d)$. If
\begin{equation}\label{eq:KK}
 \sum_{k\in\Z^d}\abs{\wh F(k)}^2\prod_{j=1}^d \log(\abs{k_j}+2)<\infty,
\end{equation}
then the rectangular partial sums $S_{\mathbf n}F$ converge to $F$ almost everywhere in the Pringsheim sense.
\end{theorem}

For $d=2$ this is due to Kaczmarz; the higher-dimensional form is due to Kojima \cite{Kaczmarz1930,Kojima1977}.

\section{Weighted differences and logarithmic coefficients}

The next elementary estimate extracts a full logarithmic power from a one-dimensional translation difference.

\begin{lemma}[]\label{lem:oscillation}
Let $d\ge1$ and $0<\eta<1$. There exists $c=c(d,\eta)>0$ such that, for every integer $m$ with $\abs{m}\ge1$,
\begin{equation}\label{eq:oscillation}
 \int_0^\eta \abs{e^{imt}-1}^2\left(\log\frac et\right)^{d-1}\frac{\dd t}{t}
 \ge c\,\log^d(\abs{m}+2).
\end{equation}
\end{lemma}

\begin{proof}
By symmetry it is enough to consider $m=N\ge1$. For every finite range $1\le N\le N_0(d,\eta)$ the assertion is absorbed by decreasing the constant, since the left-hand side is positive for each fixed $N\ge1$. Hence we assume that $N$ is large.

Choose $0<\alpha<\pi/8$. For integers $q\ge0$ set
\[
 I_q=\bigl\{t>0:\abs{Nt-(2q+1)\pi}<\alpha\bigr\}.
\]
On $I_q$ we have $\abs{e^{iNt}-1}^2\ge c_0$. If $0\le q\le Q$, where $Q=\lfloor c_1N\eta\rfloor$ and $c_1>0$ is chosen sufficiently small, then $I_q\subset(0,\eta)$ for all large $N$. Moreover, on $I_q$ one has $t\asymp(q+1)/N$ and $\abs{I_q}\asymp N^{-1}$. Hence
\begin{align*}
 \int_0^\eta \abs{e^{iNt}-1}^2\left(\log\frac et\right)^{d-1}\frac{\dd t}{t}
 &\ge C\sum_{q=0}^Q \frac1{q+1}
 \left(\log\frac{eN}{C(q+1)}\right)^{d-1}  \\
 &\ge C\int_1^{c_2N}\left(\log\frac{eN}{Cu}\right)^{d-1}\frac{\dd u}{u}.
\end{align*}
The change of variables $v=\log(eN/(Cu))$ shows that the last integral is bounded from below by $C(\log N)^d-C_{d,\eta}$. This is at least $c(d,\eta)\log^d(N+2)$ for all sufficiently large $N$, and the remaining finite range has already been absorbed into the constant.
\end{proof}

The next proposition is the main reduction step. It converts the endpoint Dini control of coordinate translation differences into the product-logarithmic Fourier coefficient condition needed for the Kaczmarz--Kojima convergence theorem.

\begin{proposition}[]\label{prop:difference-to-product}
Let $d\ge2$ and $f\in L^2(\T^d)$. Suppose that, for some $\eta\in(0,1)$ and every coordinate $j=1,\ldots,d$,
\begin{equation}\label{eq:coordinate-dini}
 \int_0^\eta \norm{\Delta_j(t)f}_{2}^{2}\left(\log\frac et\right)^{d-1}\frac{\dd t}{t}<\infty .
\end{equation}
Then
\begin{equation}\label{eq:product-log-condition}
 \sum_{k\in\Z^d}\abs{\wh f(k)}^2\prod_{j=1}^d\log(\abs{k_j}+2)<\infty .
\end{equation}
\end{proposition}

\begin{proof}
Fix a coordinate $j$. Since $f\in L^2(\T^d)$,
\[
 \widehat{\Delta_j(t)f}(k)=(e^{ik_jt}-1)\wh f(k),\qquad k\in\Z^d,
\]
and Plancherel's theorem gives, for every $t$,
\[
 \norm{\Delta_j(t)f}_2^2
 =\sum_{k\in\Z^d}\abs{\wh f(k)}^2\abs{e^{ik_jt}-1}^2.
\]
Multiplying by $(\log(e/t))^{d-1}/t$ and applying Tonelli's theorem, the assumption \eqref{eq:coordinate-dini} implies
\begin{equation}\label{eq:tonelli-coordinate}
 \sum_{k\in\Z^d}\abs{\wh f(k)}^2
 \int_0^\eta \abs{e^{ik_jt}-1}^2\left(\log\frac et\right)^{d-1}\frac{\dd t}{t}<\infty .
\end{equation}
Lemma~\ref{lem:oscillation} yields
\begin{equation}\label{eq:one-coordinate-weight}
 \sum_{\substack{k\in\Z^d\\ \abs{k_j}\ge1}}\abs{\wh f(k)}^2\log^d(\abs{k_j}+2)<\infty .
\end{equation}
Because $f\in L^2(\T^d)$, Plancherel gives $\sum_k\abs{\wh f(k)}^2<\infty$. Therefore the terms with $k_j=0$ may be added to \eqref{eq:one-coordinate-weight}, and for every $j$,
\begin{equation}\label{eq:all-coordinate-weight}
 \sum_{k\in\Z^d}\abs{\wh f(k)}^2\log^d(\abs{k_j}+2)<\infty .
\end{equation}
For non-negative $a_1,\ldots,a_d$,
\[
 \prod_{j=1}^d a_j\le \frac1d\sum_{j=1}^d a_j^d.
\]
Taking $a_j=\log(\abs{k_j}+2)$, multiplying by $\abs{\wh f(k)}^2$, and summing over $k$ gives \eqref{eq:product-log-condition} by \eqref{eq:all-coordinate-weight}.
\end{proof}

\section{Proofs of the endpoint results}

\begin{proof}[Proof of Theorem~\ref{thm:main}]
For each coordinate $j$ and every $t>0$,
\[
 \norm{\Delta_j(t)f}_2\le \omega_2(f,t).
\]
Thus the endpoint Dini assumption \eqref{eq:main-dini} implies \eqref{eq:coordinate-dini} for every coordinate. Proposition~\ref{prop:difference-to-product} gives the Kaczmarz--Kojima product-log condition \eqref{eq:product-log-condition}. The convergence conclusion follows from Theorem~\ref{thm:KK}.
\end{proof}

\begin{proof}[Proof of Corollary~\ref{cor:double-log}]
For small $t$, the hypothesis \eqref{eq:double-log-assumption} gives
\[
 \omega_2(f,t)^2\left(\log\frac et\right)^{d-1}\frac1t
 \le
 C\frac1{t\log(e/t)\left(\log\log(e^e/t)\right)^{1+2\gamma}}.
\]
With $u=\log(e/t)$, the right-hand side is integrable near the origin, since the corresponding integral is bounded by a constant multiple of
\[
 \int^{\infty}\frac{\dd u}{u(\log u)^{1+2\gamma}}<\infty .
\]
Hence \eqref{eq:main-dini} holds, and Theorem~\ref{thm:main} applies.
\end{proof}

\begin{corollary}[Recovery of Zhizhiashvili's $L^2$ theorem]\label{cor:recover-zh}
Let $d\ge2$ and $f\in L^2(\T^d)$. If \eqref{eq:zhizhiashvili-classical} holds for some $\beta>d/2$, then $S_{\mathbf n}f(x)\to f(x)$ almost everywhere in the Pringsheim sense.
\end{corollary}

\begin{proof}
After the change of variables $u=\log(e/t)$, the integral in \eqref{eq:main-dini} is dominated by a constant multiple of
\[
 \int^{\infty}u^{d-1-2\beta}\dd u,
\]
which is finite precisely when $\beta>d/2$. The assertion follows from Theorem~\ref{thm:main}.
\end{proof}

\begin{remark}
Corollary~\ref{cor:recover-zh} shows that the classical theorem of Zhizhiashvili is contained in the endpoint Dini criterion. Corollary~\ref{cor:double-log} shows the strict endpoint gain: at the critical logarithmic exponent, a summable iterated-logarithmic improvement is sufficient. This is the precise sense in which Theorem~\ref{thm:main} strengthens Zhizhiashvili's logarithmic modulus criterion.
\end{remark}

\section*{Conflicts of Interest}

The authors declare that they have no conflicts of interest.

\section*{Data Availability}

Not applicable.

\end{document}